\documentclass[letterpaper, 10 pt, conference]{ieeeconf}  % Comment this line out if you need a4paper

\IEEEoverridecommandlockouts                              % This command is only needed if 
\usepackage{graphicx} % for pdf, bitmapped graphics files
\usepackage{amsmath} % assumes amsmath package installed
\usepackage{amssymb}  % assumes amsmath package installed

\usepackage{cite}
\usepackage{xcolor}

\newtheorem{definition}{Definition}

\title{\LARGE \bf
Mutual Dependence: A Novel Method for Computing Dependencies \\Between Random Variables
}

\author{%
	Rahul Agarwal, Pierre Sacr\'{e}, and Sridevi V. Sarma% <-this stops a space
	\thanks{R. Agarwal, P. Sacr\'{e,} and S.V. Sarma are with the Institute for Computational Medicine and the Department of Biomedical Engineering, The Johns Hopkins University, Baltimore, MD (rahul.jhu@gmail.com, p.sacre@jhu.edu, sree@jhu.edu).}%
	}

\usepackage{bm}

\DeclareMathOperator{\sinc}{sinc}
\DeclareMathOperator*{\argmax}{arg\,max}
\newtheorem{theorem}{Theorem}[section]

\renewcommand{\iff}{\Leftrightarrow}

\begin{document}

\ifpdf
\DeclareGraphicsExtensions{.pdf, .jpg, .tif}
\else
\DeclareGraphicsExtensions{.eps, .jpg}
\fi

\maketitle

\begin{abstract}
In data science, it is often required to estimate dependencies between different data sources. These dependencies are typically calculated using Pearson's correlation, distance correlation, and/or mutual information. 
% need
However, none of these measures satisfy \emph{all} the Granger's axioms for an ``ideal measure''. One such ideal measure, proposed by Granger himself, calculates the Bhattacharyya distance between the joint probability density function (pdf) and the product of marginal pdfs. We call this measure the \emph{mutual dependence}. However, to date this measure has not been directly computable from data.  
% task
In this paper, we use our recently introduced maximum likelihood non-parametric estimator for band-limited pdfs, to compute the mutual dependence directly from the data.
% object
We construct the estimator of mutual dependence and compare its performance to standard measures (Pearson's and distance correlation) for different known pdfs by computing convergence rates, computational complexity, and the ability to capture nonlinear dependencies. 
% findings
Our mutual dependence estimator requires fewer samples to converge to theoretical values, is faster to compute, and captures more complex dependencies than standard measures. 
% conclusion
% perspectives
\end{abstract}

%%%%%%%%%%%%%%%%%%%%%%%%%%%%%%%%%%%%%%%%%%%%%%%%%%%%%%%%%%%%%%%%%%%%%%%%%%%%%%%%
\section{Introduction}
% context
In data science and modeling, it is often required to test whether two random variables are independent. Out of several measures that quantify dependencies between random variables \cite{Shannon1948,Lee-Rodgers1988,Szekely2007,Szekely2009}, the most widely used are mutual information $I$, Pearson's correlation $r$, and distance correlation $R$. 

% need
Mutual information, $I,$ is generally thought of as a benchmark for quantifying dependencies between random variables; however, it can only be computed by first estimating the joint and marginal probability density functions (pdfs). 
Pearson's correlation, $r$, can be directly estimated from data, but it does not capture nonlinear dependencies. 
Distance correlation, $R$, can also be directly estimated directly from data and can capture nonlinear dependencies, but is in general slow to compute (computational complexity $\mathcal{O}(n^2)$). Further, distance correlation often does not reflect the nonlinear dependencies correctly as described succinctly by R\'{e}nyi's axioms \cite{Renyi1959}, which were slightly improved upon by Granger, Maasoumi and Racine \cite{Granger2004}. See Table \ref{tb:axioms}. Specifically, distance correlation is not invariant under strictly monotonic transformations (6th axiom in Table \ref{tb:axioms}). 

An ``ideal measure'' should satisfy axioms given in Table~\ref{tb:axioms} and should be directly estimable from the data. A less popular and unnamed measure uses the Bhattacharyya distance between the joint pdf and the product of the marginals as a measure for dependence between two random variables~\cite{Kailath1967,Beran1977}. It has been shown that this measure satisfies \emph{all} six axioms. Importantly, this measure is invariant under continuous and strictly increasing transformations \cite{Skaug1996,Granger2004}. It is also closely related to mutual information, k-class entropy and copula \cite{Genest1986,Nelsen1999,Havrda1967}. In this paper, we call this measure \emph{mutual dependence},~$d$.

%Can you show the mutual dependence meaure-the formula here

\begin{table*}[t]
	\centering
	\caption{Desired properties of ideal dependency measure $\delta(X,Y)$.} \label{tb:axioms}
	\begin{tabular}{llllll}
		\# & Property & $r$ & $I$ & $R$ & $d$ \\
		\hline\\
		%$\delta(X,Y)$ & \checkmark & \checkmark & \checkmark &\checkmark \\
		1 & $\delta(X,Y) = \delta(Y,X)$ & \checkmark & \checkmark & \checkmark & \checkmark \\
		2 & $\delta(X,Y)=0$ iff $X$ and $Y$ are independent & & \checkmark & \checkmark & \checkmark \\
		3 & $0\leq\delta(X,Y) \leq 1$ & & & \checkmark & \checkmark \\
		4 & $\delta(X, Y ) = 1$ if there is a strict dependence between $X$ and $Y$ & & & \checkmark & \checkmark \\
		5 & $\delta(X, Y ) = f(r(X, Y ))$ if the joint distribution of $X$ and $Y$ is normal &\checkmark &\checkmark & \checkmark & \checkmark \\
		6 & $\delta(\psi_1(X), \psi_2(Y )) = \delta(X, Y )$ & & \checkmark  & & \checkmark
	\end{tabular}
	
	\flushleft \hspace{87.5pt}
	\tiny{Note: $\psi_1, \psi_2$ are strictly monotonic functions}
\end{table*}

Mutual dependence has not been widely used because, like mutual information, it requires non-parametric density estimation to compute the marginal and joint pdfs, which are then substituted into the theoretical measure and numerically integrated to yield estimates. This process is both computationally complex and inaccurate. 
% task
In this paper, we develop an estimator that estimates mutual dependence directly from the data. It uses our recently proposed Band-Limited Maximum Likelihood (BLML) estimator that maximizes the data likelihood function over a set of band-limited pdfs with known cut-off frequency $f_c$. The BLML estimator is consistent, efficiently computable, and results in a smooth pdf \cite{Agarwal2015}. The BLML estimator also has a faster rate of convergence and reduced computational complexity over other widely used non-parametric methods such as kernel density estimators. Along with these properties, if the BLML estimator is substituted into the expression for mutual dependence (see \eqref{dependency_eq}), the mutual dependence can be computed directly from the data without performing numerical integration, which is often inaccurate and inefficient. 

%refer to the equation with the expression..ok i se you have it later..refer to it..else the integration effects are vague

We show through simulations that $d$ converges faster than $R$ and $r$ for various data sets with different types of linear and nonlinear dependencies, and the convergence rate for computing $d$ is maintained for different type of nonlinearities. $d$ is faster to compute than $R$ as it has $\mathcal{O}(B^2+n)$ time complexity, where $B$ is the number of bins containing a finite number of samples which is always less than or equal to $n$ (the number of data samples).   

The paper is organized as follows.
Section~\ref{sec:motivation} discusses variation in different measures as a function of mutual information and nonlinearity.
Section~\ref{sec:metric} introduces the notion of mutual dependence and its estimator.
Section~\ref{sec:performance} uses simulation to compare convergence of mutual dependence with Pearson's and distance correlation for different nonlinearity dependencies and marginal pdfs. We end the paper with conclusions and future work in Section~\ref{sec:conclusions}.
 
\begin{figure}[t]
	\centering
	\includegraphics[width=0.5\textwidth]{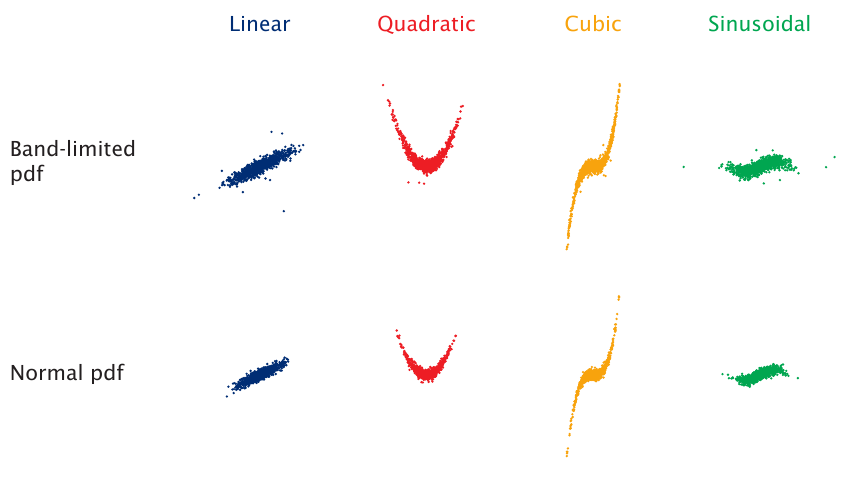}
	\caption{\textbf{Point clouds.} Illustrating point cloud for data generated from \eqref{genrating_eq} for different nonlinearities $g(x)$ and generating pdfs. $\rho=0.9$ was used for generating this data.}
	\label{fig:data}
\end{figure}
 
%%%%%%%%%%%%%%%%%%%%%%%%%%%%%%%%%%%%%%%%%%%%%%%%%%%%%%%%%%%%%%%%%%%%%%%%%%%%%%%%
\section{A motivating example} \label{sec:motivation}

%\textcolor{red}{Difference between $X$ and $Y$ and $x$ and $y$?? X is the random variable, x is the argument of the pdf...}

Consider two random variables $X$ and $Y$ defined as:
\begin{subequations}\label{genrating_eq}
	\begin{align*}
 		X& =V, \\
 		Y& =\rho \, g(X) + \sqrt{1-\rho^2} \, U,
 	\end{align*}
\end{subequations}
where $U$ and $V$ follow either a band-limited pdf 
\begin{equation*}
	f_X(x)=\frac{3}{4}\left[\sinc^4(0.2\,x-0.1)+\sinc^4(0.2\,x +0.1)\right]
\end{equation*}  
% (Figure \ref{fig:data} top row) 
or a normal pdf 
\begin{equation*}
	f_X(x)=\frac{1}{\sigma\sqrt{2\,\pi}} e^{-\frac{(x-\mu)^2}{2\,\sigma^2}}
\end{equation*} 
% (Figure \ref{fig:theoretical_prf} bottom row).
and where $g(X)$ is one of four types of (nonlinear) dependence among 
\begin{equation*}
	\text{$X$, $X^2$, $X^3$, or $\sin(X)$}.
\end{equation*} 
The `spread' $\rho$ is varied from $0.1$ to $0.9$ to obtain different degrees of dependencies. Figure \ref{fig:data} illustrates the data generated in this example.
 
 %motivate why you care to compare to mutual information here..
The goal of the dependency measures is to quantify dependencies between $X$ and $Y$ given the data. In cases where underlying pdfs are known these dependency are captured pretty nicely by mutual information.    
 
Therefore in Figure \ref{fig:theoretical_prf}, we plot theoretical values for Pearson's and distance correlation of dependence as a function of mutual information for the four different nonlinearity types and the two different generating pdfs. 
\begin{itemize}
\item Mutual information
\begin{align}
I=\int \log\left(\frac{f_{XY}(x,y)}{f_X(x)f_Y(y)}\right) f_{XY}(x,y) \mathrm{d}x\mathrm{d}y
\end{align}
\item Pearson's correlation
\begin{align}
r=\frac{E(XY)-E(X)E(Y)}{\sqrt{E(X^2)-E(X)^2}\sqrt{E(Y^2)-E(Y)^2}} .
\end{align}
\item Distance correlation
\begin{align}
\mathrm{dCov}^2(X,Y)& =\int\frac{|\phi_{XY}(s,t)-\phi_X(s)\phi_Y(t)|^2}{|s|^{1+p}|t|^{1+q}} \mathrm{d}s\mathrm{d}t \nonumber \\
R &=\frac{\mathrm{dCov}(X,Y)}{\sqrt{\mathrm{dCov}(X,X) \mathrm{dCov}(Y,Y)}}
\end{align}
here $\phi_{XY}$, $\phi_X$, $\phi_Y$ are the respective characteristic functions. $p$ and $q$ are the dimension of $X$ and $Y$. For details see \cite{Szekely2007}. (Note we have eliminated the constants $c_p,c_q$ from the definition of $\mathrm{dCov}$ as they are not needed to define $R$).

\end{itemize}

%you did not define p and q in formula for dCOV

%\textcolor{red}{shall we define measures?}
Both Pearson's and distance correlation measures depend largely on the nonlinearity for a given value of mutual information. This variability may occur because both the types of correlation measures are not invariant to strictly monotonic transformations, unlike mutual information. Therefore, changing the type of nonlinearity results in different values for both Pearson's and distance correlation, while the mutual information remains invariant. Such variance is undesirable as it may lead to incorrect inferences when comparing dependencies between data having different types of nonlinear dependencies. Therefore, a measure that is invariant to strictly monotonic transformations is desirable.    
 
\begin{figure}
	\centering
	\includegraphics[width=0.5\textwidth]{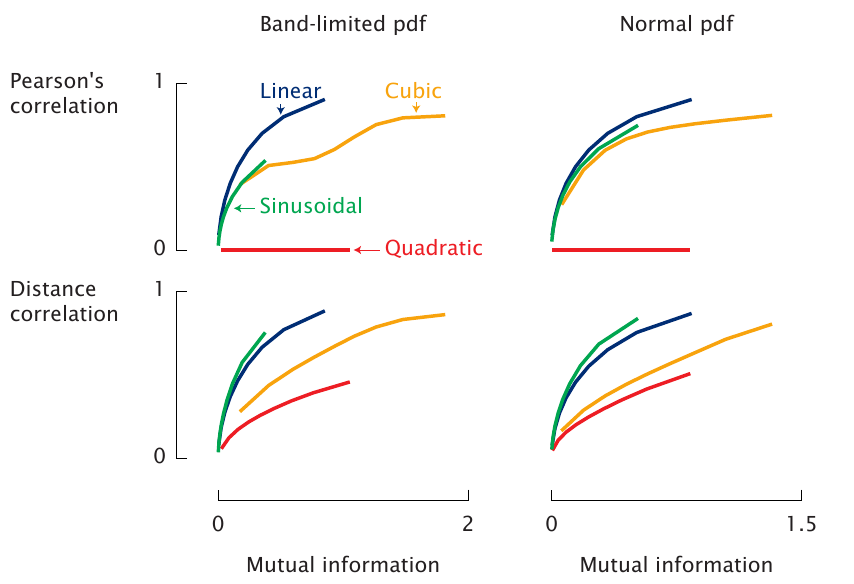}
	\caption{\textbf{Pearson's and distance correlation.} Illustrating theoretical values of $r$ and $R$ as a function of $I$ for different nonlinearities and generating pdfs as used in Figure \ref{fig:data}.}
	\label{fig:theoretical_prf}
\end{figure}  
 
%%%%%%%%%%%%%%%%%%%%%%%%%%%%%%%%%%%%%%%%%%%%%%%%%%%%%%%%%%%%%%%%%%%%%%%%%%%%%%%%
\section{Mutual dependence and its estimation} \label{sec:metric}

In this section, we introduce the {\it mutual dependence}, which is based on an unnamed existing measure, and show several properties of this measure. Then, we derive an estimator of mutual dependence derived directly from data generated from band-limited pdfs. Finally, we  describe efficient algorithms to compute this estimator.

\subsection{Mutual dependence}
Consider two random variables $X$ and $Y$, their joint distribution $f_{XY}(x,y)$, and their marginal distributions $f_X(x)$ and $f_Y(y)$. 
These random variables are independent if and only if $f_{XY}(x,y) = f_X(x)\,f_Y(y)$. It is therefore natural to measure dependence as the distance (in the space of pdfs) between the joint and the product of marginal distributions. A good distance candidate is the Bhattacharyya distance (also known as Hellinger distance). See \cite{Granger2004, Skaug1996} for details.

\begin{definition}
The \emph{mutual dependence} $d(X,Y)$ between two random variables $X$ and $Y$ is defined as the Bhattacharyya distance $d_h(\cdot,\cdot)$ between their joint distribution $f_{XY}(x,y)$ and the product of their marginal distributions $f_X(x)$ and $f_Y(y)$, that is,
\begin{equation}\label{dependency_eq}
d(X,Y) \triangleq d_{h}(f_{XY}(x,y),f_X(x)\,f_Y(y))
\end{equation}
with 
\begin{equation}
	d^2_{h}(p(\mathbf{x}),q(\mathbf{x})) \triangleq \frac{1}{2}\int\left(\sqrt{p(\mathbf{x})}- \sqrt{q(\mathbf{x})}\right)^2 \mathrm{d}\mathbf{x}.
\end{equation}
\end{definition}

We call this measure `mutual dependence' as it represents mutual information most closely. For a given value of mutual information, the value of mutual dependence remains almost the same irrespective of the nonlinearity type, which is not true for Pearson's and distance correlation measures.

\subsection{Properties of mutual dependence}

Due to symmetry of $d(\cdot,\cdot)$, it is easy to see that $d(X,Y)=d(Y,X)$. The measure $d\in(0,1)$ if $X$ and $Y$ are partially dependent which quantifies the degree of dependence between the two random variables. 
In the extreme cases, $d=0$ $\iff$ $X$ and $Y$ are independent and $d=1$ if either $x$ or $y$ is a Borel-measurable function of the other.  Also, it can be easily established that $d$ is invariant under strictly monotonic transformations $\psi_1$ and $\psi_2$, i.e $d(X,Y)=d(\psi_1(X),\psi_2(Y))$. A detailed description of these properties can be found in \cite{Granger2004,Skaug1996}.

For jointly normal data, the mutual dependence can be estimated by first calculating the Bhattacharyya distance between two multivariate Gaussian distributions \cite{Pardo-Llorente2006}
\begin{align}
	d_h^2 & = 1 - \frac{\left|\Sigma_1\right|^{\frac{1}{4}} \left|\Sigma_2\right|^{\frac{1}{4}}}{\left|\frac{1}{2}\Sigma_1 + \frac{1}{2}\Sigma_2\right|^{\frac{1}{2}}} \times  \nonumber \\
& \ \ 	\exp\left(-\frac{1}{8}(\mu_1 - \mu_2)^T\left(\frac{1}{2}\Sigma_1 + \frac{1}{2}\Sigma_2\right)^{-1}(\mu_1 - \mu_2)\right)
\end{align}
where $\mu_1$ and $\mu_2$ are the mean vectors and $\Sigma_1$ and $\Sigma_2$ covariance matrices.
Then substituting
\begin{align*}
	\mu_1 &=0, & \Sigma_1 & =\left[\begin{array}{cc} \sigma^2_x & \rho \, \sigma_x \,\sigma_y \\ \rho \, \sigma_x \,\sigma_y & \sigma^2_y\end{array}\right],\\
	\mu_2 &=0, & \Sigma_2 & = \left[\begin{array}{cc} \sigma^2_x & 0 \\ 0 & \sigma^2_y\end{array}\right], 
\end{align*}
gives
\begin{align}
	d(X,Y) & = \sqrt{1 - \frac{(1-\rho^2)^{\frac{1}{4}}}{(1-\frac{1}{4}\rho^2)^{\frac{1}{2}}}} 
	\triangleq M(\rho).
\end{align}

This shows that mutual dependence satisfies axiom 5 (see Table 1). %or something like this

\begin{figure}
	\centering
	\includegraphics[width=0.5\textwidth]{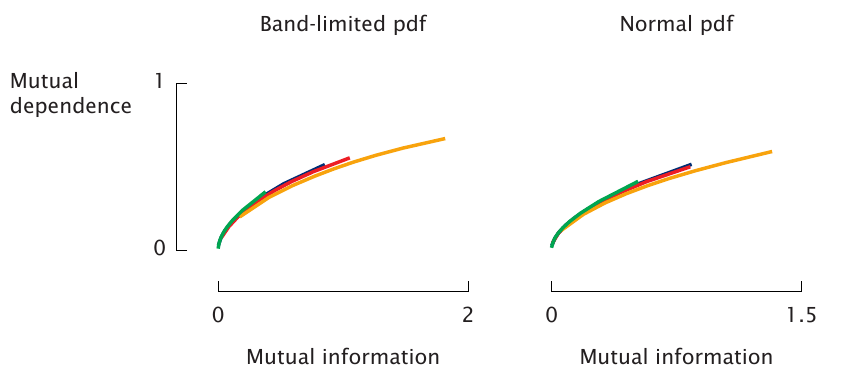}
	\caption{\textbf{Mutual dependence.} Illustrating theoretical values of $d$ as a function of $I$ for different nonlinearities and generating pdfs as used in Figure \ref{fig:data}.}
	\label{fig:theoretical_prf2}
\end{figure}

\subsection{Estimation of mutual dependence} \label{sec:estimation}
To estimate $d,$ we use the BLML method \cite{Agarwal2015} that maximizes the likelihood of observing data samples over the set of band-limited pdfs. The BLML estimator is shown to outperform kernel density estimators (KDE) both in convergence rates and computational time and hence provides a better alternative for non-parametric estimation of pdfs. In addition, the structure of the BLML estimator is well suited for evaluating the integral in \eqref{dependency_eq}, resulting in an estimate which is a direct function of observed data and hence avoids numerical integration errors. 

Below we briefly describe the BLML estimator.

\vspace{0.1in}

\begin{theorem}\label{BLML_thm} 
Consider $n$ independent samples of an unknown BL pdf, $f_X(x)$, with assumed cut-off frequency $f_c.$ Then the BLML estimator of $f_X(x)$ is given as:
\begin{equation}\label{BLML_eq}
\hat{f}_X(x)=\left(\frac{1}{n}\sum_{i=1}^n\hat{c}_i \sinc_\bold{f_c}({\bf x-x}_i)\right)^2
\end{equation}
where, $\bold{f_c}\in \mathbb{R}^m$ is the assumed cutoff frequency, vectors $\bold{x}_i$'s, with $i=1\cdots n$, are the data samples, $\sinc_\bold{f_c}(\bold{x})\triangleq \prod_{k=1}^{m}\frac{\sin(\pi f_{ck}x_k)}{\pi x_k}$ and the vector $\bold{\hat{c}}\triangleq \left[ \hat{c}_1, \cdots, \hat{c}_n \right]^T$, is given by
\begin{equation}\label{BLMLgen_eq2}
\hat{\mathbf{c}}=\argmax_{ \bm{\rho_n}(\bold{c)=0} }\left(\prod \frac{1}{c_i^2}\right).
\end{equation}
\noindent
Here $\rho_{ni}(\bold{c})\triangleq \sum_{j=1}^n c_js_{ij}-\frac{n}{c_i}$ with $ s_{ij}\triangleq \sinc_\bold{f_c}({\bf x}_i-{\bf x}_j)$.
\end{theorem} 

See~\cite{Agarwal2015} for details. Now we introduce the estimator for $d$, $\hat{d}$ in the following theorem.

\begin{figure*}
	\centering
	\includegraphics[width=0.975\textwidth]{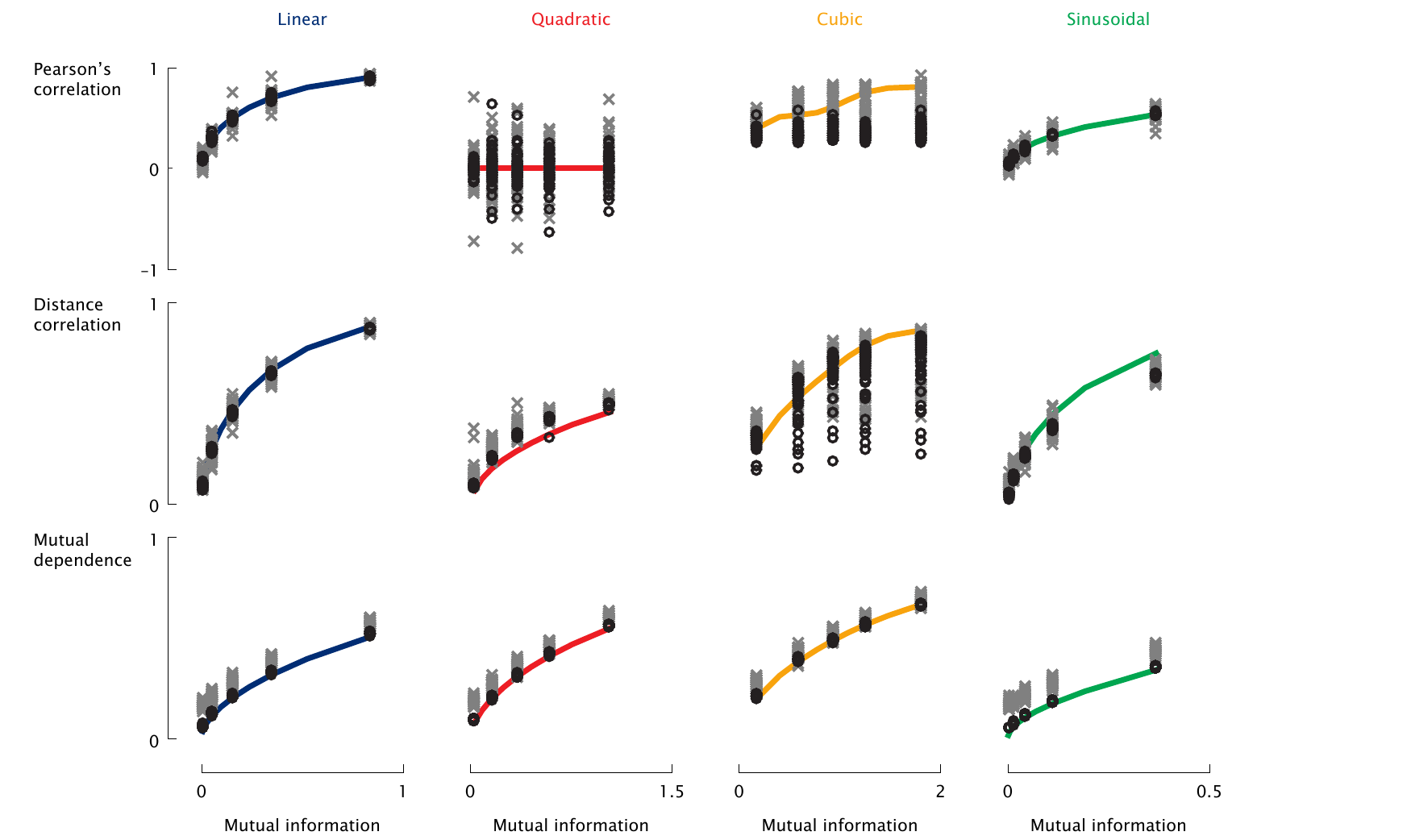}
	\caption{\textbf{Monte Carlo Estimates for band-limited generating pdfs.} The Monte Carlo distribution of estimates for different measures for different nonlinearities and band-limited generating pdfs. $\times$ess mark the estimates calculated using sample sizes $n=316$ whereas $\circ$s mark the estimates calculated using sample size $n=10000$. $d$ is estimate assuming the cut-off frequency $f_c=\frac{1}{1-\rho^2}$.}
	\label{fig:convergencebl}
\end{figure*}

\begin{figure*}
	\centering
	\includegraphics[width=0.975\textwidth]{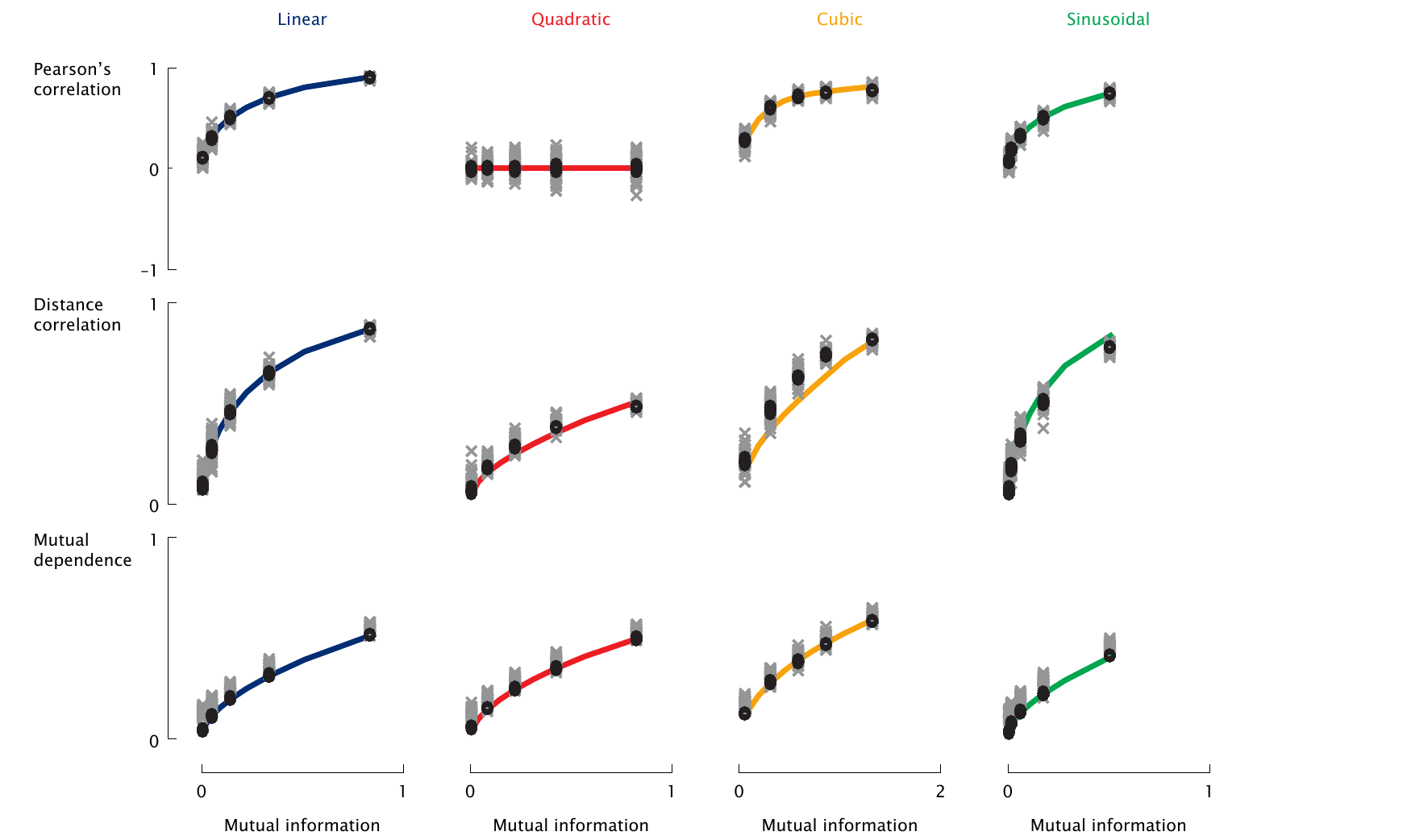}
	\caption{\textbf{Monte Carlo Estimates for normal generating pdfs.}  The Monte Carlo distribution of estimates for different measures for different nonlinearities and normal generating pdf. $\times$es mark the estimates calculated using sample sizes $n=316$ whereas $\circ$s mark the estimates calculated using sample size $n=10000$. $d$ is estimate assuming the cut-off frequency $f_c=\frac{1}{1-\rho^2}$.}
	\label{fig:convergencenormal}
\end{figure*} 

\begin{theorem}
If $(x_i,y_i)$ \  $i=1,\cdots,n$ are $n$ paired independent and identically distributed data observations and $f_c$ is the cut-off frequency parameter. Then the estimator for mutual dependence is given as:

\begin{equation}
\hat{d}\triangleq d_h(\hat{f}_{XY},\hat{f}_X\hat{f}_Y)=\sqrt{1-\frac{1}{n}\sum\frac{\hat{c}_i^{(XY)}}{\hat{c}_i^{(X)}\hat{c}_i^{(Y)}}}
\end{equation}
\noindent
where $\hat{\bold{c}}^{(XY)}=\{\hat{c}_i^{(XY)}\}_{i=1}^n$ is given by:
\begin{subequations}
\begin{align*}
\hat{\bold{c}}^{(XY)} &=\argmax_{ \bm{\rho_n^{(XY)}}(\bold{c)=0} }\left(\prod \frac{1}{c_i^2}\right) \\
\rho_{ni}^{(XY)}(\bold{c}) & \triangleq \sum_{j=1}^n c_j\frac{\sin(\pi f_{c}(x_i-x_j)}{\pi (x_i- x_j)} \frac{\sin(\pi f_{c}(y_i-y_j))}{\pi(y_i-y_j)}-\frac{n}{c_i}, 
\end{align*}
\end{subequations}
\noindent
$\hat{\bold{c}}^{(X)}=\{\hat{c}_i^{(X)}\}_{i=1}^n$  is:
\begin{subequations}
\begin{align*}
\hat{\bold{c}}^{(X)} & =\argmax_{ \bm{\rho_n^{(X)}}(\bold{c)=0} }\left(\prod \frac{1}{c_i^2}\right) \\
\rho_{ni}^{(X)}(\bold{c}) & \triangleq \sum_{j=1}^n c_j\frac{\sin(\pi f_{c}(x_i-x_j)}{\pi (x_i- x_j)}-\frac{n}{c_i}
\end{align*}
\end{subequations}
\noindent
and $\hat{\bold{c}}^{(Y)}=\{\hat{c}_i^{(Y)}\}_{i=1}^n$ is: 
\begin{subequations}
\begin{align*}
\hat{\bold{c}}^{(Y)} & =\argmax_{ \bm{\rho_n^{(Y)}}(\bold{c)=0} }\left(\prod \frac{1}{c_i^2}\right) \\
\rho_{ni}^{(Y)}(\bold{c}) & \triangleq \sum_{j=1}^n c_j \frac{\sin(\pi f_{c}(y_i-y_j))}{\pi(y_i-y_j)}-\frac{n}{c_i}. \nonumber
\end{align*}
\end{subequations}

\end{theorem}
\begin{proof}
The BLML estimators of  $f_{XY}$, $f_X$ and $f_Y$ from Theorem \ref{BLML_thm} (using same cut-off frequency $[f_c, f_c], f_c$ and $f_c$ respectively) are plugged into \eqref{dependency_eq} and the resultant equation is integrated which gives $\hat{d}$.
\end{proof}

\subsection{Computation of mutual dependence}
As described in \cite{Agarwal2015} solving for $\hat{\bold{c}}$ requires exponential time. Therefore, heuristic algorithms also described in  \cite{Agarwal2015} such as \texttt{BLMLBQP} and \texttt{BLMLTrivial}, can be used directly to compute $c^{(XY)}_i$, $c^{(X)}_i$, $c^{(Y)}_i$ approximately for each $i$ for small scale ($n<100$) and large scale ($n>100$) problems, respectively. 

To further improve the computational time \texttt{BLMLQuick} algorithm \cite{Agarwal2015} can also be used. \texttt{BLMLQuick} uses binning  and estimates $c^{(XY)}_i$, $c^{(X)}_i$, $c^{(Y)}_i$ approximately for each $i$. It is also shown in \cite{Agarwal2015} that both \texttt{BLMLTrivial} and \texttt{BLMLQuick} algorithms yield consistent estimate of pdfs if the true pdf is strictly positive, therefore in cases where the joint $f_{XY}>0$ the estimate, $d$ is also consistent.

%%%%%%%%%%%%%%%%%%%%%%%%%%%%%%%%%%%%%%%%%%%%%%%%%%%%%%%%%%%%%%%%%%%%%%%%%%%%%%%%
\section{Performance of mutual dependence} \label{sec:performance}

%ok you change to past tense here "evaluated"..check tense throughout..i changed to evaluate

In this section, we evaluate the performance of our estimator for mutual information by first comparing the empirical distribution of the estimator with the empirical distribution of the estimators for Pearson's and distance correlation for different mutual information values, $I$, nonlinearities, $g(X)$,  and generating pdfs, $f_X(x).$  We compare the convergence of these metrics to the true values for different sample sizes. Finally, we compare the computational complexity of our estimator with the estimator for distance correlation to evaluate the relative computational time needed to implement each estimator. 

\subsection{Comparison of convergence rate for different nonlinearities}
Figures \ref{fig:convergencebl} and \ref{fig:convergencenormal} plot the estimated $\hat{r}$, $\hat{R}$ and $\hat{d}$ for $n=316$ and $n=10000$ from about 50 Monte Carlo runs as a function of $I$ for different nonlinearities (linear, quadratic, cubic and sinusoidal) and generating pdfs (band-limited and normal). Underlaid are the respective theoretical values. Specifically, the first row shows about 50 Monte Carlo computation of $\hat{r}$ for different $I$ values, nonlinearities and generating pdfs. It can be seen that for both $n=316$ and $n=10000$, $\hat{r}$ works best for linear and sinusoidal data, but for quadratic data $\hat{r}$ has a larger variance and for cubic data $\hat{r}$ has a larger bias in bandlimited case. The second row shows 50 Monte Carlo computations of $\hat{R}$ for different $I$ values, nonlinearities and generating pdfs. It can be seen that for both $n=316$ and $n=10000$, $\hat{R}$ works best for linear data, but for quadratic and sinusoidal data, it has larger bias whereas for cubic data it has larger variance. The bottom row shows 50 Monte Carlo computation of $\hat{d}$ for different $I$ values, nonlinearities and generating pdfs. It can be seen that $\hat{d}$ works equally good for all nonlinearities and shows less bias and variance than both $\hat{r}$ and $\hat{R}$.

Figure \ref{fig:error} plots the integration (over different $I$ values) of mean squared error (IMSE) between the theoretical and estimated measures using about 50 Monte Carlo runs, for different nonlinearities and generating pdf types.
\begin{equation}
IMSE=\int\frac{1}{m}\sum_{i=1}^{m} (\hat{\delta}(I)-\delta(I))^2 \, \mathrm{d}I
\end{equation}

Here, $m$ is the number of Monte Carlo simulations and $\delta$ is the dependency metric. It can be seen from the Figure \ref{fig:error} that the convergence rate is fastest for $\hat{d}$ irrespective of nonlinearity type and/or generating pdf. $r$ and $R$ show an equally fast convergence rate for linear and normal data, but the rate is slower for nonlinear and non-normal data. Specifically, the first row shows convergence of $\hat{r}$, from which it can be established that convergence of $\hat{r}$ to the theoretical values is fastest for linear data. For nonlinear data, the convergence is slo either due to large bias or variance as discussed previously. The second row shows convergence of $\hat{R}$. It can be seen that $\hat{R}$ does well for linear data, but the rate slows down and saturates for nonlinear data again due to either large bias or variance. Specially, for cubic and band-limited data, the IMSE of  $\hat{R}$ does not decrease with increasing the number of samples, this is due to the nondecreasing variance of the estimator (see Figure \ref{fig:convergencebl}). The bottom row shows convergence of $\hat{d}$. It can be seen that $\hat{d}$ converges equally well for all data types and generating pdfs.  
 
\begin{figure}
	\centering
	\includegraphics{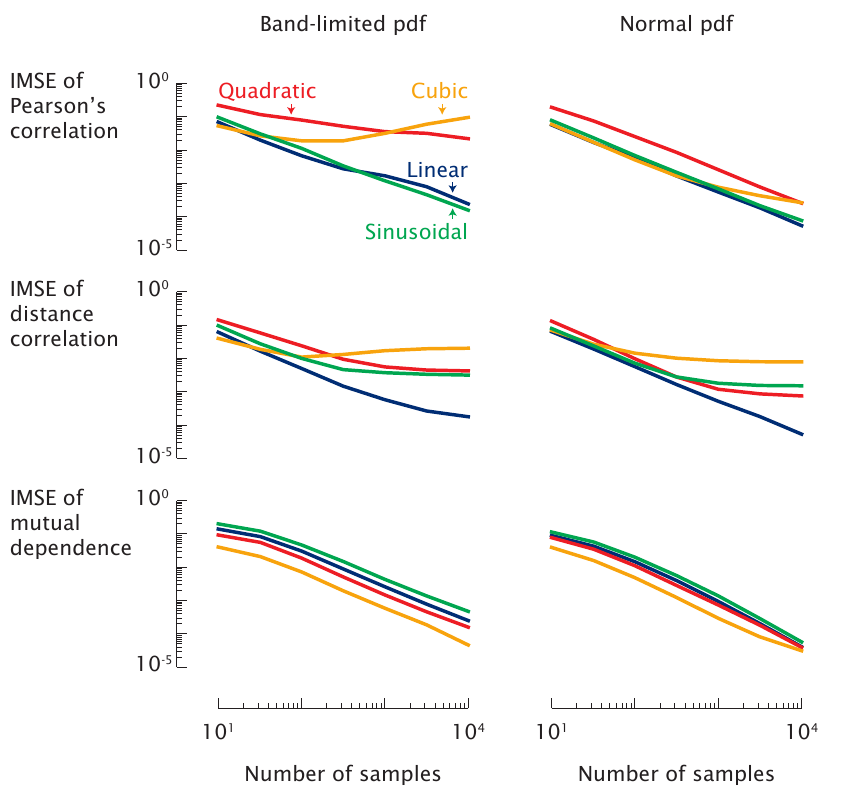}
	\caption{\textbf{Integrated mean squared error vs sample size.} Showing the Integrated mean squared error as a function of sample size $n$ for different measures, different nonlinearities and different generating pdfs. $d$ is estimate assuming the cut-off frequency $f_c=\frac{1}{1-\rho^2}$.}
	\label{fig:error}
\end{figure} 
 
\subsection{Comparison of computational time} 
 The computational complexity of computing $\hat{r}$ is least which is $\mathcal{O}(n)$, whereas computational complexity of computing $\hat{R}$ is maximum which is $\mathcal{O}(n^2)$. $\hat{d}$ is same as computational complexity of \texttt{BLMLQuick} algorithm which is $\mathcal{O}(B^2+n)$, where $B$ is the number of bins containing nonzero number of samples, which is always less than equal to $n$. For dense data $B \ll n$ therefore computation of $\hat{d}$ is a lot quicker than estimating $\hat{R}$ in such cases.  
 
\section{Conclusions}\label{sec:conclusions}
In this paper, we introduced a novel estimator for measuring dependency that can be directly computed from the data. Our estimator computes the mutual dependence which is an ``ideal'' measure for dependence between two random variables~\cite{Granger2004}. Our estimator has advantages over mutual information estimators as it does not require estimating the pdfs from data. It also has advantage over Pearson's and distance correlation estimators as it is invariant under strictly monotonic transformation. Further, we showed that under simulation, estimators of both Pearson's and distance correlation require more samples to achieve the same integrated mean squared error (IMSE) as compared to our mutual dependence estimator showing lower convergence rate. The slower convergence rate for the estimators of Pearson's and distance correlation was due to their higher variance and bias for the nonlinearly dependent data. Such nonlinearities did not affect our estimator and it showed a uniform decrease in IMSE as the sample size increases for all tested nonlinearities. Even further, our estimate for mutual dependence showed a computational time complexity of $\mathcal{O}(B^2+n)$ where $B \leq n$ is the number of bins,  which is superior to the time complexity of distance correlation ($\mathcal{O}(n^2)$) and is much faster when the data is dense. 

\subsection{Future work}
Although our estimator for the mutual dependence showed some nice properties under simulation, it remained to be established that it shows consistency for any nonlinearity which would require building up a theoretical proof. Further, in this paper, we assumed through out that we knew the cut-off frequency of the band-limited pdf or approximate cut-off frequency for the normal pdf (the band where most of the power of pdf lies, in case it is not band limited). However, in general this cut-off frequency is not known. A more in-depth analysis is needed to understand the behavior of our estimator as a function of the cut-off frequency.  

%\subsection{Wishlist for a dependence metric}

%axioms etc.

%\subsection{Mutual dependence}
%\begin{definition}
%	
%\end{definition}

%%%%%%%%%%%%%%%%%%%%%%%%%%%%%%%%%%%%%%%%%%%%%%%%%%%%%%%%%%%%%%%%%%%%%%%%%%%%%%%%
%\section{Estimation of mutual dependence} \label{sec:estimation}
%
%
%\subsection{Consistency of estimation}
%
%%%%%%%%%%%%%%%%%%%%%%%%%%%%%%%%%%%%%%%%%%%%%%%%%%%%%%%%%%%%%%%%%%%%%%%%%%%%%%%%%
%\section{Algorithms to compute mutual dependence}
%
%
%\subsection{Comparison of computational time}
%\subsection{Computation of cut-off frequency}

%%%%%%%%%%%%%%%%%%%%%%%%%%%%%%%%%%%%%%%%%%%%%%%%%%%%%%%%%%%%%%%%%%%%%%%%%%%%%%%%
%\section{Illustrations}
%
%\subsection{Motivating example}
%
%
%\subsection{Wikipedia Data}
%
%
%\begin{figure*}[t]
%	\centering
%	\includegraphics{./figures/wiki.pdf}
%	\caption{To be done....}
%	\label{fig:wiki}
%\end{figure*} 

%%%%%%%%%%%%%%%%%%%%%%%%%%%%%%%%%%%%%%%%%%%%%%%%%%%%%%%%%%%%%%%%%%%%%%%%%%%%%%%%

%clean up below

\bibliographystyle{ieeetr}
\bibliography{references}
\end{document}